\documentclass[twoside]{siamltex}

\newtheorem{example}{Example}[section]
\newcommand{\sgap}{\vspace{0.05in}}
\newcommand{\gap}{\vspace{0.1in}}

\newcommand{\rnplus}{R_+^n}

\newcommand{\Rn}{R^n}

\newcommand{\A}{{\cal A}}
\newcommand{\I}{{\cal I}}
\newcommand{\B}{{\cal B}}
\newcommand{\D}{{\cal D}}
\newcommand{\C}{{\cal C}}

\title{{\bf
{\cal Z}-tensors and complementarity problems}\thanks{This research was supported by the National Natural Science Foundation of China (11301022,11431002), and the Hong Kong Research Grant Council (Grant No.
PolyU 502111, 501212, 501913 and 15302114).}}

\author{M. Seetharama Gowda\footnotemark[2], \hspace{1mm} Ziyan Luo\footnotemark[3], \hspace{1mm}  Liqun Qi\footnotemark[4], \hspace{1mm} and Naihua Xiu\footnotemark[5]
}
\begin{document}

\maketitle

\gap

\begin{center}
{\bf October 24, 2015, Revised December 28, 2016}
\end{center}

\gap

\renewcommand{\thefootnote}{\fnsymbol{footnote}}

\footnotetext[2]{Department of Mathematics and Statistics, University of Maryland, Baltimore County, Baltimore, MD 21250. Email: gowda@umbc.edu}

\footnotetext[3]{The State Key Laboratory of Rail Traffic Control and Safety, Beijing Jiaotong University, Beijing 100044, P.R. China.  E-mail: starkeynature@hotmail.com}

\footnotetext[4]{Department of Applied Mathematics, The Hong Kong Polytechnic University, Hung Hom, Kowloon, Hong Kong. E-mail: liqun.qi@polyu.edu.hk, http://www.polyu.edu.hk/ama/people/detail/1/}

\footnotetext[5]{Department of Mathematics, School of Science, Beijing Jiaotong University, Beijing, P.R.China.  E-mail: nhxiu@bjtu.edu.cn}

\begin{abstract}
Tensors are multidimensional analogs of matrices. In this paper, based on degree-theoretic ideas, we 
study  homogeneous nonlinear complementarity problems induced by tensors. By  specializing this to 
$Z$-tensors (which are tensors with non-positive off-diagonal entries), we 
describe various equivalent
conditions for a $Z$-tensor to have the global solvability property. 
We show 
by an example that the global solvability need not imply  unique solvability and provide a  sufficient and easily checkable  condition for unique solvability. 
\end{abstract}


\section{Introduction}
A tensor is simply a multidimensional analog of a matrix. Given natural numbers $m\,(\geq 2)$ and $n$, an $m$th order, $n$-dimensional tensor is of the form
\begin{equation}\label{tensor a}
\A=[a_{i_1\,i_2\,i_3\,\cdots\,i_m}]\,
\end{equation}
where $a_{i_1\,i_2\,i_3\,\cdots\,i_m}\in R,\,\, i_1,i_2,i_3,\ldots,i_m\in \{1,~\ldots,~n\}.$
During the last decade, tensors have become very important in various areas. Numerous articles extending basic
concepts and results of matrix theory have been written, see for example, \cite{chang-pearson-zhang, ding-qi-wei, su-huang-ling-qi, kannan et al, qi, yang-yang-II, yang-yang-I, zhang-qi-zhou, zhou-qi-wu}.
With a view towards bringing in optimization ideas, researchers
have introduced various  complementarity concepts \cite{bai-huang-wang, che-qi-wei, DLQ2015, huang-qi, luo-qi-xiu, Song-Qi-2015, SQ2015, song-qi, SY2015,wang-huang-bai}.  Given a tensor $\A$ in the form (\ref{tensor a}), we define a function $F:\Rn\rightarrow \Rn$ whose $i$th component is given by
$$F_i(x):=\sum_{i_2,i_3,\ldots,\,i_m=1}^{n} a_{i\,i_2\,i_3\cdots\,i_m}x_{i_2}x_{i_3}\cdots\,x_{i_m}.$$
This function,  abbreviated by
$$F(x)=\A x^{m-1},$$
has a homogeneous polynomial of degree $m-1$ in each component.
Corresponding to this $F$ and any $q\in \Rn$, we consider the {\it tensor complementarity problem} TCP$(\A,q)$: Find $x\in \Rn$ such that
$$x\geq 0,\,F(x)+q\geq 0\,\,\mbox{and}\,\,\langle x, F(x)+q\rangle =0,$$
where $x\geq 0$ means that each component of $x$ is nonnegative, etc.
This is a generalization of the {\it linear complementarity problem} (corresponding to $m=2)$, a
special instance of a {\it nonlinear complementarity problem} and a particular case of a {\it variational inequality problem} corresponding to the closed convex cone $\rnplus$. Complementarity problems and variational inequality problems have been extensively studied and there is a vast literature dealing with
existence, uniqueness, computation, and applications, see for example, \cite{facchinei-pang-I, facchinei-pang-II}. In the last decade or so, much work has been done in extending these to symmetric cones. Since the tensor complementarity problem is a special case of a nonlinear complementarity problem, the entire theory of nonlinear complementarity problems is applicable to tensor complementarity problems. However, because each component of $F(x)$ is a homogeneous polynomial (of the same degree), we may expect some
specialized results;  see \cite{gowda-pang} for an early reference where (multi)functions with certain `homogeneity' are treated. 
In a recent article \cite{huang-qi}, Huang and Qi formulate 
an  $n$-person noncooperative game (where the utility
function of every player is a homogeneous polynomial of degree) as a tensor complementarity problem. 

\gap

In the first part of the article, we discuss the global solvability property in tensor complementarity problems. We say that a tensor $\A$ is a {\bf Q}-tensor \cite{SQ2015} if TCP$(\A,q)$ has a solution for all $q$. Using topological degree theory, we present a basic global solvability result (Theorem  \ref{nonzero degree impies Q}).
In the second part, we study $Z$-tensors, which are tensors with non-positive `off-diagonal' entries. It is easy to see that such a tensor can be written as
$$\A=r\I-\B,$$
where $r\in R$, $\I$ is the identity tensor and $\B$ is a nonnegative tensor (that is, all its entries are nonnegative).
Properties of nonnegative tensors, particularly in relation to the Perron-Frobenius theorem, have been explored in several recent papers, see
\cite{chang-pearson-zhang, yang-yang-I, yang-yang-II, zhou-qi-wu}.
If $\rho(B)$ denotes the spectral radius of $\B$, one says that the $Z$-tensor
$\A=r\I-\B$ is an $M$-tensor if $r\geq \rho(B)$ and strong (or nonsingular) $M$-tensor if  $r> \rho(B)$.
Some properties of $M$-tensors and strong $M$-tensors have been discussed in
\cite{ding-qi-wei, zhang-qi-zhou, kannan et al}. Motivated by a paper of Luo et al. \cite{luo-qi-xiu}, here, we undertake a study of complementarity properties of $Z$-tensors, specifically asking when a $Z$-tensor $\A$ has the $Q$-property.
 In addition to proving several equivalent
properties, we show how degree theory offers a way of understanding the solvability of certain equations arising in $Z$-tensor complementarity problems.\\

This paper is organized as follows. In Section 2, we recall some results about nonnegative tensors. Section 3 covers a basic result about $Q$-tensors via degree theory. In Section 4, we consider $Z$-tensors and characterize the strong $M$-tensor property in various equivalent ways. Finally, in Section 5, we describe some refined properties of $Z$-tensors such as the surjectivity of the map $F$, 
 $P$ and globally uniquely solvable properties.

\section{Preliminaries} Throughout this paper,
$\Rn$ denotes the $n$-dimensional Euclidean space with the usual inner product. The usual  and the infinity norms are respectively denoted by $||\cdot||$ and $||\cdot||_\infty$. For $x\in \Rn$, we write $x\geq 0$ $(x>0)$ if all components of $x$ are nonnegative (respectively, positive). The nonnegative orthant of $\Rn$ is denoted by $\rnplus.$ We denote the complex $n$-space by $C^n$. \\
Let $\A=[a_{i_1\,i_2\,i_3\,\cdots\,i_m}]$ denote an $m$th order, $n$-dimensional tensor. The entries $a_{i\,i\,\cdots\,i}$, $1\leq i\leq n$, are the `diagonal' entries; the rest are `off-diagonal' entries of $\A$. The identity tensor is one with all diagonal entries one and off-diagonal entries zero.
A tensor is said to be {\it nonnegative} if all its entries are nonnegative.

\gap

A complex number $\lambda$ is said to be an {\it eigenvalue} of $\A$ if there exists a nonzero vector $x\in C^n$ such that
$$\A x^{m-1}=\lambda\,x^{[m-1]},$$
where $x^{[m-1]}$ is the vector in $C^n$ with $i$th component $x_i^{m-1},$
see \cite{chang-pearson-zhang, qi}. Define the spectrum $\sigma(\A)$ to be the set of all eigenvalues of $\A$. Then, the spectral radius of $\A$ is defined by
$$\rho(\A):=\max\{ |\lambda|:\lambda\in \sigma(\A)\}.$$

\gap

The following is a  Perron-Frobenius type theorem for nonnegative tensors.

\sgap

\begin{proposition} (\cite[Theorem 2.3]{yang-yang-I}) If $\B$ is a nonnegative tensor, then $\rho(\B)$ is an eigenvalue of $\B$ with a nonnegative eigenvector.
\end{proposition}

\gap

Next, we recall a Collatz-Wielandt type result.

\sgap

\begin{proposition}(\cite[Lemma 5.3 and Theorem 5.3]{yang-yang-I})\label{yang-yang proposition} Let $\B$ be a nonzero nonnegative $m$th order, $n$-dimensional tensor. Let $\rho(\B)$ be its spectral radius. Then, for any $d>0$,
$$\min_{i} \frac{(\B d^{m-1})_i}{d_i^{m-1}}\leq \rho(\B)\leq \max_{i}  \frac{(\B d^{m-1})_i}{d_i^{m-1}}.$$
Moreover,
$$\rho(\B)=\max_{0\neq x\geq 0}\,\min_{x_i>0}\frac{(\B x^{m-1})_i}{x_i^{m-1}}.$$
\end{proposition}

\gap

Given an $m$th order, $n$-dimensional tensor $\A$, let $I\subseteq \{1,2,\ldots, n\}$. Then, the principal subtensor of $\A$ corresponding to $I$ is given by
$\widetilde{\A}:=[a_{i_1\,i_2\,i_3\,\cdots\,i_m}],$
where $i_k\in I$ for all $k=1,2\ldots,m.$

\gap

The following corollary is immediate from the above proposition.

\sgap

\begin{corollary} \label{rho of a subtensor}
Let $\D$ be a principal subtensor of a nonnegative tensor $\B$. Then $\rho(\D)\leq \rho(\B).$
\end{corollary}

\gap

Let $\A$ be a $Z$-tensor written in the form $\A=r\I-\B$, where $r\in R$ and $\B$ is a nonnegative tensor. We say that $\A$ is an $M$-tensor if $r\geq \rho(\B)$ and a {\it strong $M$-tensor} if $r>\rho(\B)$.
The following result and its proof are modified versions of Theorem 3.3 in \cite{zhang-qi-zhou}.

\gap

\begin{proposition} \label{positive stable proposition}
Let $\A=r\I-\B$ be a $Z$-tensor and $\mu(\A):= \min_{\lambda\in \sigma(\A)}\, Re(\lambda).$ Then,
$$\mu(\A)=r-\rho(\B).$$
Moreover, $\mu(\A)$ is a real eigenvalue of $\A$ corresponding to a real eigenvector.
\end{proposition}
\begin{proof} Since $\rho(\B)$ is a real eigenvalue of $\B$ corresponding to a real eigenvector, $r-\rho(\B)$ is a real eigenvalue of $\A$ corresponding to (the same) real eigenvector. Hence,
$$\mu(\A)\leq r-\rho(\B).$$
On the other hand, if $\lambda\in \sigma(\A)$, then $r-\lambda\in \sigma(\B)$ and so,
$$r-Re(\lambda)\leq |r-\lambda|\leq \rho(\B).$$
This yields $r-\rho(\B)\leq Re(\lambda)$ and (taking the minimum over all $\lambda\in \sigma(\A)$), $r-\rho(\B)\leq \mu(\A).$
\end{proof}
\section{$Q$-tensors}
Generalizing the concept of a $Q$-matrix of linear complementarity theory \cite{cottle-pang-stone},
$Q$-tensors were introduced  in \cite{SQ2015}. Let $\A$ be an $m$th order, $n$-dimensional tensor and
$F(x):=\A x^{m-1}.$
{\it We say that $\A$ is a $Q$-tensor if for every $q\in \Rn$, TCP$(\A,q)$ has a solution.} Using the notation $\min\{x,y\}$ for the componentwise minimum of two vectors $x$ and $y$ (that is, $(\min\{x,y\})_i:=\min\{x_i,y_i\})$, we note 
that $x$ is a solution of TCP$(\A,q)$ if and only if $x$ is a solution of the piecewise polynomial equation
$$\min\{x,F(x)+q\}=0.$$
 Moreover, when $m$ is even, the same $x$ is also a solution of
$$\min\{x^{[m-1]},F(x)+q\}=0.$$
(Note that $\min\{x^{[m-1]},F(x)\}$ is homogeneous of degree $m-1$, while
$\min\{x,F(x)\}$ may not be homogeneous.)

\gap

In what follows, we employ degree theoretic ideas. All necessary ideas and results concerning degree theory are given in \cite{facchinei-pang-I}, Prop. 2.1.3.
Here is a short review. Suppose $\Omega$ is a bounded open set in $\Rn$, $g:\overline{\Omega}\rightarrow \Rn$ is continuous and $p\not\in g(\partial\,\Omega)$, where $\overline{\Omega}$ and $\partial\,\Omega$ denote, respectively, the closure and boundary of $\Omega$.
Then the degree of $g$ over $\Omega$ with respect to $p$ is defined; it is an integer and will be denoted  by 
$\mathrm{deg}\,(g,\Omega, p)$. When this degree is nonzero, the equation $g(x)=p$ has a solution in $\Omega$. If $h:\Rn \rightarrow \Rn$ is another continuous function such that $\sup_{x\in \overline{\Omega}}||g(x)-h(x)||_{\infty}$ is sufficiently small, then
$\mathrm{deg}\,(g,\Omega, p)=\mathrm{deg}\,(h,\Omega, p)$. This is the {\it nearness property} of the degree. 
Suppose $g(x)=p$ has a unique solution, say, $x^*$ in $\Omega$. Then, $\mathrm{deg}\,(g,\Omega^\prime,p)$ is  constant over all bounded open sets $\Omega^\prime$ containing $x^*$ and contained in $\Omega$. This common degree is called the
{\it local  degree} of $g$ at $x^*$;
it will be denoted by $\mathrm{deg}\,(g,x^*)$. In particular, if $h:\Rn \rightarrow \Rn$ is a continuous function such that
$h(x)=0\Leftrightarrow x=0$, then, for any bounded open set $\Omega$ containing $0$, we have
$$\mathrm{deg}\,(h,0)=\mathrm{deg}\,(h,\Omega,0);$$ moreover, when $h$ is the identity map, $\mathrm{deg}\,(h,0)=1$.
Let $H(x,t):\Rn\times [0,1]\rightarrow \Rn$ be continuous (in which case, we say that $H$ is a homotopy) and the zero set
$\{x:\,H(x,t)=0\,\,\mbox{for some}\,\,t\in [0,1]\}$
be  bounded. Then, for any bounded open set $\Omega$ in $\Rn$ that contains this zero set, we have the
{\it homotopy invariance of degree}:
$$\mathrm{deg}\,\Big (H(\cdot,1),\Omega,0\Big )=\mathrm{deg}\,\Big (H(\cdot,0),\Omega,0\Big ).$$

The following is a basic result dealing with tensor complementarity problems.
Given a tensor $\A$, let
$$\Phi(x):=\min\{x,F(x)\}.$$

\gap

\begin{theorem}\label{nonzero degree impies Q}
Suppose that
$$\Phi(x)=0\Rightarrow x=0\quad\mbox{and}\quad \deg(\Phi,0)\neq 0.$$
Then, $\A$ is a $Q$-tensor and TCP$(\A,q)$ has a nonempty compact solution for any $q\in\Rn$.
\end{theorem}
\begin{proof} We fix a bounded open set $\Omega$ in $\Rn$ containing zero so that 
$\mathrm{deg}\,(\Phi,0)=\mathrm{deg}\,(\Phi,\Omega, 0).$ For any $q\in \Rn$, let $\Phi_q(x):=\min\{x, F(x)+q\}.$ Since 
$||\Phi_q(x)-\Phi(x)||_{\infty}\leq ||q||_{\infty}$, we see that $\sup_{x\in \overline{\Omega}}|||\Phi_q(x)-\Phi(x)||_{\infty}$ is small when $q$ is sufficiently close to zero. Thus, by the nearness property of degree (see Prop. 2.1.3(c), \cite{facchinei-pang-I}), for all $q$ sufficiently close to zero,
$\deg(\Phi_q,\Omega, 0)= \deg(\Phi,\Omega,0)=\deg(\Phi,0)\neq 0.$ This means that $\Phi_q(x)=0$ has a solution, that is, TCP$(\A,q)$ has a solution for all $q$ near zero. Since $F(x)$ is positive homogeneous of degree $m-1$, by scaling, TCP$(\A,q)$ has a solution for all $q\in \Rn$. Now we will show the compactness of the solution set of TCP$(\A,q)$ for any given $q\in\Rn$ under the condition $\Phi(x)=0\Rightarrow x=0$.
To see this, first observe that  the solution set of TCP$(\A,q)$ is closed as it is the same as that of $\min\{x,\,F(x)+q\}=0$. The boundedness of the solution set is seen via a `normalization argument'
as follows. Suppose, if possible, for some $q$, the solution set of  $\min\{x,\, F(x)+q\}=0$ is unbounded. Let $x^{(k)}$ be a sequence in the solution set with $||x^{(k)}||\rightarrow \infty$. Writing $\min\{x,\,F(x)+q\}=0$ in terms of complementarity conditions, we see that $\min\{\lambda\,x,\,F(\lambda\,x)+\lambda^{m-1}q)\}=\min\{\lambda\,x,\,\lambda^{m-1}(F(x)+q)\}=0$ for all $\lambda>0$. Now replacing $x$ by $x^{k}$, choosing
$\lambda:=||x^{(k)}||^{-1}$, letting $k\rightarrow \infty$, and putting (without loss of generality) $\overline{x}:=\lim \frac{x^{(k)}}{||x^{(k)}||}$, we get $\min\{\overline{x}, F(\overline{x})\}=0$. We reach a contradiction as $||\overline{x}||=1$ and at the same time $\overline{x}=0.$ Thus, the nonempty solution set of TCP$(\A,q)$ is closed and bounded, hence compact.
\end{proof}

\gap

\noindent{\bf Remarks. }
 The condition $\Phi(x)=0\Rightarrow x=0$, which is equivalent to TCP$(\A,0)$ having zero as the only solution, has been shown to be equivalent to the $R_0$-property of $\A$, see  Proposition 3.1 (i) in \cite{SQ2015}. The boundedness of the involved solution set of the tensor complementarity problem has been addressed under this $R_0$-property in Theorem 3.3, \cite{SY2015}. The $Q$-property of $\A$ is discussed in Theorem 3.2 of \cite{SQ2015} in which the $R$-property of $\A$ is required. Here, differing from the $R$-property, the $Q$-property is achieved via degree theory. 
\gap

\begin{corollary}\label{conditions giving Q}
Under each of the following conditions, $\A$ is a $Q$-tensor and the corresponding TCP$(\A,q)$ has a nonempty compact solution for any $q\in\Rn$.
\begin{enumerate}
\item [(i)] There exists a vector $d>0$ such that
for TCP$(\A,0)$ and TCP$(\A,d)$,  zero (vector) is the only solution.
\item [(ii)] $\A$ is a strictly semi-monotone (or strictly semi-positive) tensor, that is, for each nonzero $x\geq 0$, $\max_{i}\,x_i(\A x^{m-1})_i>0.$
\item [(iii)] $\A$ is a strictly copositive
 tensor, that is, for all $0\neq x\geq 0$,
$\A x^m:=\langle \A x^{m-1},x\rangle>0$.
\item [(iv)] $\A$ is a positive definite tensor, that is, for all $x\neq 0$,
$\A x^m:=\langle \A x^{m-1},x\rangle>0$.
\end{enumerate}
\end{corollary}
\begin{proof} $(i)$ Note that this condition is precisely what is given in the well-known Karamardian's theorem. Our degree theory proof offers, in addition to existence, a stability result (in the sense that certain nonhomogeneous
nonlinear complementarity problems of the form NCP$(G,p)$ with $(G,p)$ close to $(F,0)$ will also have solutions). Now to show that condition $(i)$ implies the desired results, we set up a homotopy:
$$H(x,t):=\min\{x,F(x)+td\}\quad 0\leq t\leq 1.$$
Since TCP$(\A,0)$ and TCP$(\A,d)$ have zero solutions, we have
$H(x,0)=0\Rightarrow x=0$ and $H(x,1)=0\Rightarrow x=0.$ In addition, for any $t$,
$0<t<1$, $H(x,t)=0$ implies, by scaling and using the homogeneity of $F$,
$\min\{sx,F(sx)+d\}=0$ for some positive $s$. This yields $x=0$. Thus, the zero set of the entire homotopy reduces to just $\{0\}.$ Now, by the homotopy invariance of the degree,
$$\deg(\Phi,0)=\deg(H(x,0),0)=\deg(H(x,1),0).$$
As $H(x,1)=\min\{x,F(x)+d\}=x$ near zero, we see that $\deg(H(x,1),0)=1.$
Thus, $\deg(\Phi,0)=1$. Now the above theorem shows that $\A$ is a $Q$-tensor and TCP$(\A,q)$ has a nonempty compact solution for any $q\in\Rn$.\\
When condition $(ii)$ holds, for any $d>0$, TCP$(\A,0)$ and TCP$(\A,d)$ have zero solutions. Hence $\A$ is a $Q$-tensor and the corresponding TCP$(\A,q)$ has a nonempty compact solution for any $q\in\Rn$.\\
It is easy to see that $(iv)\Rightarrow (iii)\Rightarrow (ii)$. Thus, the asserted conclusions hold when condition $(iii)$ or $(iv)$ holds.
\end{proof}

\gap

\noindent{\bf Remarks. }
It is worth pointing out that (i) of Corollary \ref{conditions giving Q} is actually equivalent to the $R$-property of $\A$ as discussed in Proposition 3.1 (ii) and Theorem 3.2 in \cite{SQ2015}, and the $Q$-property under (ii) of Corollary \ref{conditions giving Q} has been discussed in Corollary 3.3 in \cite{SQ2015}. Besides, see \cite{che-qi-wei}, where conditions $(iii)$ and $(iv)$ are discussed in relation to the $Q$-property of $\A$.

\section{$Z$-tensors; Some basic results} In this section, we characterize the $Q$-property of a $Z$-tensor in various equivalent ways. We start by recalling a result that says that in the case of a complementarity problem corresponding to a
$Z$-tensor, feasibility implies solvability.

\gap

\begin{proposition} (Corollary 1, \cite{luo-qi-xiu}). Suppose $\A$ is a $Z$-tensor. If TCP$(\A,q)$ is feasible, that is, there exists $u\geq 0$ such that $\A u^{m-1}+q\geq 0$, then it is solvable.
\end{proposition}

Based on this  proposition, we can characterize $Z$-tensors having the $Q$-property.

\gap

\begin{theorem}
Suppose $\A$ is an $m$th order, $n$-dimensional tensor. Consider the following statements:
\begin{enumerate}
\item [(i)] $\A$ is a $Q$-tensor.
\item [(ii)] For every $q\in \Rn$, TCP$(\A,q)$ is feasible.
\item [(iii)] There exists $d>0$ such that $\A d^{m-1}>0$.
\end{enumerate}
Then, $(i)\Rightarrow (ii)\Leftrightarrow (iii)$. Moreover,  these statements are equivalent when $\A$ is a $Z$-tensor.
\end{theorem}
\begin{proof} The implication $(i)\Rightarrow (ii)$ is obvious.\\
The equivalence of $(ii)$ and $(iii)$ is established in Theorem 3.2 in \cite{SY2015}.
When $\A$ is a $Z$-tensor, we quote the previous proposition to see that TCP$(\A,q)$ is solvable under $(ii)$.
\end{proof}

\gap

The following result characterizes the $Q$-property of a $Z$-tensor in different ways. These conditions/properties have been discussed in various articles. We collect them together here, offer a proof for completeness and for further refined results (in the next section).
Note that these results are generalizations of similar results for $Z$-matrices. They are also similar to the ones for $Z$-transformations on proper cones \cite{gowda-tao}.

\gap

\begin{theorem}\label{big theorem}
Let $\A$ be a $Z$-tensor given by $\A=r\I-\B$, where $r\in R$ and $B$ is a nonnegative tensor. Then the following statements are equivalent:
\begin{enumerate}
\item [(a)] $\A$ is a $Q$-tensor.
\item [(b)] For each $q\geq 0$, there exists $x\geq 0$ such that $\A x^{m-1}=q$.
\item [(c)] $\A$ is an $S$-tensor, that is, there exists $d>0$ such that $\A d^{m-1}>0$.
\item [(d)] $\A$ is a strong $M$-tensor, that is, $r>\rho(\B)$.
\item [(e)] $\mbox{For all}\,\,0\neq x\geq 0,\, \max_{i} \,x_i(\A x^{m-1})_i>0.$
\item [(f)] For all $q\geq 0$, zero  is the only solution of TCP$(\A,q)$.
\end{enumerate}
In addition, the above  conditions are further equivalent to
\begin{enumerate}
\item [(i)] $\A$ is positive stable, that is, $\mu(\A)>0$.
\item [(ii)] For all $\varepsilon\geq 0$, $(\A +\varepsilon\,\I)x^{m-1}=0\Rightarrow x=0.$
\item [(iii)] For any nonnegative diagonal tensor $\D$ compatible with $\A$, $\A+\D$ is a strong $M$-tensor.
\end{enumerate}
\end{theorem}
\begin{proof}
$(a)\Rightarrow (b)$: Assume $(a)$ and let $q\geq 0$. Then there exists $x\in \Rn$ such that
$$x\geq 0,\,y:=\A x^{m-1}-q\geq 0,\,\langle x,y\rangle =0.$$
Then, by the $Z$-property of $\A$, $\langle \A x^{m-1},y\rangle \leq 0$. This yields $\langle y+q,y\rangle \leq 0$ and $||y||^2+\langle q,y\rangle \leq 0$. As $q\geq 0$, we get $y=0$ showing $\A x^{m-1}=q$.\\
$(b)\Rightarrow (c)$: Taking (any) $q>0$, we get an $x\geq 0$ such that $\A x^{m-1}=q>0$. By continuity, there exists $d>0$ such that $\A d^{m-1}>0$.\\
$(c)\Rightarrow (d)$: This comes from Proposition \ref{yang-yang proposition}.\\
$(d)\Rightarrow (e)$:
Assume $(d)$ and suppose there exists a nonzero $x$ with $x\geq 0$ and $x_i(\A x^{m-1})_i\leq 0$ for all $i$.
without loss of generality, let $I=\{i:x_i\neq 0\}$ be $\{1,2,\ldots,l\}$. Then, $(\A x^{m-1})_i\leq 0$ for all $i\in I$.
Since $\A=r\I-\B$, considering a principal subtensor $\D$ of $\B$, we get
$(r\I-\D)\,y^{m-1}\leq 0,$ where $y$ is the vector formed by the $x_i$, $i\in I$. This leads to
$r\leq \frac{(\D y^{m-1})_i}{y^{m-1}_i}$ for all $i\in I$ and hence to $r\leq \rho(\D)$. As $\rho(\D)\leq \rho(\B)$, this clearly is a contradiction.
Hence we have $(d)\Rightarrow (e)$.\\
$(e)\Rightarrow  (f)$:
let $q\geq 0$ and let $x$ be a solution of
TCP$(\A,q)$. If $x$ is nonzero, then $x_i(\A x^{m-1})_i>0$ for some $i$ and
$x_i(\A x^{m-1}+q)_i>0$. Thus, $x$ cannot be complementary to $\A x^{m-1}+q$,
yielding a contradiction.
\\
$(f)\Rightarrow (a)$:
This comes from Corollary \ref{conditions giving Q}, Item (i) by taking $q=0$ and $q>0$ in $(f)$.\\
Now for the additional statements:\\
$(d)\Leftrightarrow (i)$: This comes from Proposition \ref{positive stable proposition}.\\
$(i)\Rightarrow (ii)$:
If $(ii)$ is false, then $\A$ will have a non-positive real eigenvalue, contradicting $(i)$.\\
$(ii)\Rightarrow (i)$: If $\mu(\A)\leq 0$, then $\varepsilon:=-\mu(\A)$ will satisfy $(\A+\varepsilon\,\I)x^{m-1}=0$ for some nonzero $x$.
\\
$(e)\Rightarrow (iii)$: Let $\D$ be any nonnegative diagonal tensor $\D$ (compatible with $\A$). Clearly,
$\A+\D$ is a $Z$-tensor. Suppose there is a
 nonzero nonnegative $x$, with $x_i\left [(\A+\D) x^{m-1}\right ]_i\leq 0$ for all
 $i$.
Then,  $x_i(\A x^{m-1})_i\leq 0$ for all $i$, contradicting $(e)$. Thus, $\A+\D$ satisfies a condition similar to $(e)$ and hence a strong $M$-tensor.
\\
The implication $(iii)\Rightarrow (e)$ holds by taking $\D=0$.
\end{proof}

\gap

\noindent{\bf Remarks. }
  The equivalence of $(a)$ and $(c)$ can also be seen by the previous theorem.  The equivalence of $(e)$ and $(f)$ is also given in Theorem 3.2 of \cite{song-qi}. When $\A$ is a strong $M$-tensor, combining Items $(ii)$ and $(iii)$, we get: For any nonnegative diagonal tensor $\D$,
\begin{equation} \label{strong m-tensor plus diagonal}
(\A+\D)x^{m-1}=0\Rightarrow x=0.
\end{equation} 

\gap

\section{Some refined results for $Z$-tensors}
When the $Z$-tensor $\A$ is a matrix (corresponding to $m=2$), there are more than 52 conditions equivalent to $\A$ being a strong $M$-matrix \cite{berman-plemmons}. Some generalizations of these were considered in Theorem \ref{big theorem}. In what follows, we
prove some refined results for even ordered tensors.

\gap

\noindent{\bf Surjectivity of the map $F(x):=\A x^{m-1}$.}\\
In Theorem \ref{big theorem}, Item $(b)$, we saw that when $\A$ is a strong $M$-tensor,  the equation $F(x)=q$ has a
solution for every $q\geq 0$. This raises the question whether this is true for all $q\in \Rn$. When $m=2$, $F$ is a linear map. In this case, the solvability of $F(x)=q$ for all $q\geq 0$ implies that the image of $F$ contains an open set and hence gives the  surjectivity of $F$. As $F$ is linear,  this gives the  injectivity of $F$ and consequently,  the invertibility of $F$.  Additionally, $F^{-1}(\Rn_+)\subseteq \Rn_+$. This fails when $m$ is odd: Take $m=3$, $n=2$, $\A=\I$ and consider the $F(x)=(x_1^2,x_2^2)^{\top}$. Clearly, $F(x)=q$ is solvable for all $q\geq 0$, but not for all $q\in R^2$. However, we have the following result for even ordered tensors and a related example.

\gap

\begin{theorem}\label{surjectivity theorem}
Suppose $\A$ is a $Z$-tensor of even order. Consider the following statements: 
\begin{enumerate}
\item [(a)] $\A$ is a strong $M$-tensor.
\item [(b)]  $F(x)=0\Rightarrow x=0\quad\mbox{and}\quad \deg(F,0)=1.$
\item [(c)] $F(x)$ is surjective.
\end{enumerate}
Then, $(a)\Rightarrow (b)\Rightarrow (c)$. 
 \end{theorem}
\begin{proof}
$(a)\Rightarrow (b)$: Suppose  $\A$ is a strong $M$-tensor. Then, by Item $(ii)$ of Theorem \ref{big theorem}, $F(x)=0\Rightarrow x=0$. Thus, the local degree of $F$ at the origin is defined. Let $\A=r\I-\B$, where $\B$ is nonnegative tensor with $\rho(\B)<r$. Then,
for any $t\in [0,1]$, $r\I-t\B$ is also a strong $M$-tensor. Thus,
$(r\I-t\B ) x^{m-1}=0\Rightarrow x=0$. Now the homotopy
$$H(x,t):=(r\I-t\B ) x^{m-1}$$
connects $H(x,0)=r x^{[m-1]}=:G(x)$ and $H(x,1)=F(x)$; moreover,
$H(x,t)=0\Rightarrow x=0$ for all $t\in [0,1]$. Thus,
by the homotopy invariance of degree,
$$\deg(F,0)=\deg(G,0).$$
As $m$ is even, the local degree of the
one variable function $\phi(t)=t^{m-1}$ at zero is one; it follows from Cartesian product property of  degree (see Prop. 2.1.3(h) in \cite{facchinei-pang-I})
that $\deg(G,0)=1$.
Hence,  $\deg(F,0)=1$. \\
$(b)\Rightarrow (c)$: Assume $(b)$ and let $\Omega$ be a bounded open set in $\Rn$ containing zero. Then, by the nearness property of the degree, for all $q$ close to zero, $\deg(F-q,\Omega,0)=\deg(F,\Omega,0)=\deg(F,0)=1$. This means that the equation $F(x)-q=0$ has a solution for all such $q$.
Since $F$ is positive homogeneous, by scaling, we see that $F(x)-q=0$ will have a solution for all $q\in \Rn$. This proves the surjectivity of $F$. 
\end{proof}

\gap

The following example shows that the map $F$ in the above theorem need not be injective and that the inclusion $F^{-1}(\Rn_+)\subseteq \Rn_+$ may not hold.

\gap

\begin{example} \label{example} Let $\A=[a_{i_1i_2i_3i_4}]$ be of order $4$ and dimension $2$ with
$$ a_{1111}=a_{2222}=1, ~a_{1112}=-2, ~a_{1122}=-\alpha, ~\textrm{other~entries ~}0,$$
where $\alpha\in \{0,4\}$.
Obviously, $\A$ is a $Z$-tensor. For this tensor,
$$(\A x^3)_1=F_1(x)= x_1^3-2x_1^2x_2-\alpha\,x_1x_2^2\quad\mbox{and}\quad (\A x^3)_2=F_2(x)=x_2^3,$$
where  $x=(x_1,x_2)^{\top}\in R^2$.
Now, for any nonzero $x=(x_1,x_2)^{\top}$, we have:
\begin{enumerate}
  \item if $x_2\neq 0$, then $x_2\left(\A x^3\right)_2=x_2^4>0$;
  \item if $x_2=0$ (in which case $x_1\neq 0$), then $x_1\left(\A x^3\right)_1=x_1^4>0$.
\end{enumerate}
\end{example}
Thus, condition $(e)$ of Theorem \ref{big theorem} holds;  hence, $\A$ is a strong $M$-tensor.
\\
When $\alpha=4$, $F$ equals $(1,1)^{\top}$  at $(-1,1)^{\top}$ and at $(t,1)^{\top}$ for some $t>0$.
This means that $F$ is not injective and the inclusion
$F^{-1}(R^2_+)\subseteq R^2_+$ does not hold.

\gap

\noindent{\bf The $P$-property}\\
It is well known that a $Z$-matrix has the $P$-property if and only if it is a strong $M$-matrix \cite{berman-plemmons}. Does such a
statement hold for $Z$-tensors? Recall that for a square real matrix $A$, the $P$-property can be described in any one of the following three equivalent ways \cite{cottle-pang-stone}:

\begin{enumerate}
\item [(i)] Every principal minor of $A$ is positive.
\item [(ii)] For each nonzero $x\in \Rn$, $\max_{i}\,x_i(Ax)_i>0$.
\item [(iii)] For every $q\in \Rn$, LCP$(A,q)$ has a unique solution.
\end{enumerate}

\gap

We will show below that for strong $M$-tensors, appropriate analogs of $(i)$ and $(ii)$ hold, but  $(iii)$ may fail.

Now for the positive principal minor property.
While the determinant of a tensor is defined (see \cite{su-huang-ling-qi}), it is not clear how to relate the (positive) determinants with the $Z$-property. So,  we describe the positive principal minor property in a different way. Suppose  $A$ is an invertible matrix. Then, $f(x):=Ax$ is linear and
$f(x)=0\Rightarrow x=0$. Thus, $\deg(f,0)$ is defined and moreover
$\deg(f,0)=sgn \,\det(A)=1$ if and only if the determinant of $A$ is positive. A similar statement holds for principal submatrices of $A$ as well. Thus, we may interpret
the positive principal minor property of $A$ by saying that $f_{\alpha}(y)=0\Rightarrow y=0$ and
$\deg(f_{\alpha},0)=1$ for each $f_{\alpha}$ that corresponds to a principal submatrix of $A$. We now state a generalization of this to even order $Z$-tensors.
The one dimensional example $\A=[1]$ with $m=3$, $n=1$ and $F(x)=x^2$ shows that the result fails for odd order tensors.
\gap

\begin{theorem}
Suppose $\A$ is a $Z$-tensor of even order. Then 
for every  principal subtensor  $\widetilde{\A}$ of $\A$, the corresponding function
$\widetilde{F}(x):=\widetilde{\A} x^{m-1}$ satisfies the conditions
$$\widetilde{F}(x)=0\Rightarrow x=0\quad\mbox{and}\quad \deg(\widetilde{F},0)=1.$$
\end{theorem}
\begin{proof}   Let $\A=r\I-\B$, where $\B$ is a nonnegative tensor and $r>\rho(B).$
The case of $\A$ and $F(x)=\A x^{m-1}$ has been dealt with in the previous theorem. We assume that $\widetilde{\A}$ is a subtensor of $\A$, not equal to $\A$. Then there exists a proper subset $I$ of $\{1,2,\ldots,n\}$, which we assume without loss of generality, $I=\{1,2,\ldots, l\}$  such that
$$\widetilde{\A}=[a_{j_1j_2\cdots j_m}],$$
where $j_k\in I$ for all $k=1,2,\ldots,m.$
Let $\D$ be the subtensor of $\B$ corresponding to
this $I$ so that $\C:=\widetilde{\A}=r\I-\D$. As  $\D$ is a principal subtensor of $\B$, we must have $\rho(\D)\leq \rho(\B)<r$. Thus, $\C$ is a strong $M$-tensor. By what has been proved earlier, for $G(x)=\C x^{m-1}$, $x\in R^l$,
$G(x)=0\Rightarrow x=0$ and $\deg(G,0)=1.$ 
This completes the proof.
\end{proof}

\gap

We now consider the $P$-matrix condition $(ii)$: for each nonzero $x\in \Rn$, $\max_{i}\,x_i(Ax)_i>0$.
Recently, Song and Qi \cite{Song-Qi-2015} extended this to tensors: A tensor $\A$ is said to be a $P$-tensor if
for any nonzero $x\in \Rn$, $\max_{i}\,x_i(\A x^{m-1})_i>0$. This was further extended in  \cite{DLQ2015}:
A tensor $\A$ is said to be an (extended) $P$-tensor if for any  nonzero $x$,  $\max_{i}\,x^{m-1}_i(\A x^{m-1})_i>0$.

\gap

\begin{theorem}\label{M-P} Suppose $\A$ is a $Z$-tensor.   Then the following are equivalent:
\begin{enumerate}
\item [(a)] $\A$ is a strong $M$-tensor.
\item [(b)] For any nonzero $x$,  $\max_{i}\,x^{m-1}_i(\A x^{m-1})_i>0$.
\end{enumerate}
If $m$ is even, these are further equivalent to
\begin{enumerate}
\item [(c)] For any nonzero $x$, $\max_{i}\,x_i(\A x^{m-1})_i>0$.
\end{enumerate}
\end{theorem}
\begin{proof} $(a)\Rightarrow (b)$: This implication  comes from Proposition 4.1 in \cite{DLQ2015}, whose proof is
based on $H$-tensors and diagonal dominance ideas. Here, for completeness, we provide a (slightly different) proof.
We prove the implication by induction on $n$. The result is clearly true for $n=1$.
Suppose $(a)$ holds and $(b)$ fails for some nonzero $x$:  $x^{m-1}_i(\A x^{m-1})_i\leq 0$ for all $i$. Such a condition cannot hold for any proper principal subtensor of $\A$ by our induction hypothesis. Thus, no component of $x$ can be zero.
Then, by putting $\alpha_i:=\frac{(\A x^{m-1})_i}{x^{m-1}_i}$, we see that each $\alpha_i$ is nonpositive. Let $\D$ be a nonnegative diagonal tensor with diagonal components $-\alpha_i$, so that $(\A+\D)x^{m-1}=0$.
As $\A$ is a strong $M$-tensor, (\ref{strong m-tensor plus diagonal}) shows that this cannot happen. Thus, $(a)\Rightarrow (b)$.\\
$(b)\Rightarrow (a)$: If condition $(b)$ holds for all nonzero $x$, it certainly holds for all nonzero nonnegative $x$. Consider such an $x$. Then there exists $x_i>0$ with $x^{m-1}_i(\A x^{m-1})_i>0$ or equivalently, $x_i(\A x^{m-1})_i>0$. This implies condition $(e)$ in Theorem \ref{big theorem}. Thus, $\A$ is a strong $M$-tensor.
\\
Now suppose that $m$ is even. Then the signs of $x_i$ and $x_i^{m-1}$ are the same. Consequently, $(b)$ and $(c)$ are equivalent.
\end{proof}

\gap

\noindent{\bf Remark. }  When $m$ is odd, $(a)$ may not imply $(c)$:
Take $\A=[1]$ with $m=3$, $n=1$ and $F(x)=x^2$.

\gap

\noindent{\bf Global uniqueness}\\
As noted previously, for a matrix $A$, the linear complementarity problem LCP$(A,q)$ has a unique solution for all $q\in \Rn$ if and only if $A$ is a $P$-matrix. In particular, this global uniqueness property holds for a strong $M$-matrix.
To see what happens for tensors, consider the strong $M$-tensor $\A$ of Example \ref{example} with $\alpha=0.$ By Theorem \ref{M-P}, $\A$ is actually an (extended) $P$-tensor.
For $q=(0,-1)^{\top}$, we have two solutions to  TCP$(\A,q)$, namely, $(0,1)^{\top}$ and $(2,1)^{\top}$. Thus, {\it uniqueness of TCP solution may not prevail even for strong $M$-tensors (or for extended $P$-tensors).}
This raises the question: {\it which strong $M$-tensors admit unique solutions in all related tensor complementarity problems}? In the complementarity literature,
a function $f:\Rn\rightarrow \Rn$ is said to have the {\it Globally Uniquely Solvable} property (GUS-property for short) if for all $q\in \Rn$,
the nonlinear complementarity problem NCP$(f,q)$ has a unique solution.
Two well-known conditions implying the GUS-property are: The strong monotonicity condition (see Section 2.3 in \cite{facchinei-pang-I}) and the `positively bounded Jacobians' condition of Megiddo  and Kojima (see Lemma 1, \cite{che-qi-wei}).
The GUS-property in the context of tensor complementarity problems has been addressed recently in
 \cite{che-qi-wei}. In the result below, we offer an easily checkable sufficient condition for a strong $M$-tensor to have the GUS-property.

\gap

\begin{theorem}\label{GUS} Suppose $\A=[a_{i_1\cdots i_m}]$ is a strong $M$-tensor of order $m\,(\geq 3)$ and dimension $n$  such that for each index $i$,
$$a_{i\,i_2\cdots i_m}=0\quad\mbox{ whenever} \,\, i_j\neq i_k\quad\mbox{for some}\,\,j\neq k.$$
Then, for any $q\in \Rn$, TCP($\A,q$) has a unique solution.
\end{theorem}
\begin{proof} As $\A$ has the  $S$-tensor property (see Theorem \ref{big theorem}), it  follows from \cite[Theorem 3]{ding-qi-wei} and \cite[Proposition 5]{ding-qi-wei} that there exists a positive diagonal matrix $D=Diag(d_i)\in R^{n\times n}$ such that the tensor $\bar{\A}=\A D^{m-1}:=[\bar{a}_{i_1\cdots i_m}]$, defined by
\begin{equation}\label{bar-A}\bar{a}_{i_1\cdots i_m}=a_{i_1\cdots i_m}d_{i_1}\cdots d_{i_m},~~\forall \,i_1,\cdots,i_m\in \{1,\cdots,n\}\end{equation}
is  strictly diagonally dominant; in fact,
\begin{equation}\label{strict} \bar{a}_{i\cdots i}>\sum\limits_{i_2,\cdots,i_m}|\bar{a}_{ii_2\cdots i_m}| - \bar{a}_{i\cdots i} =-\sum\limits_{k\neq i}\bar{a}_{ik\cdots k},~~\forall\, i\in\{1,\cdots,n\},
\end{equation}
Now we claim that for any given $q\in\Rn$,  TCP($\bar{\A},q$) has a unique solution.
As $\bar{\A}$ is a $Z$-tensor with
$\bar{\A}e^{m-1}>0$, where $e$ is the vector of ones in  $\Rn$, it follows that $\bar{\A}$ is a strong $M$-tensor; hence,
 TCP($\bar{\A},q$) has a solution. To prove uniqueness, assume that  that there exist distinct solutions $\hat{y}$ and $\tilde{y}$ of TCP($\bar{\A},q$). That is, for any $i\in\{1,\cdots,n\}$,
\begin{equation}
\left\{
  \begin{array}{ll}
    \hat{y}_i\geq 0, ~\left(\bar{\A}\hat{y}^{m-1}+q\right)_i\geq 0,~\hat{y}_i\left(\bar{\A}\hat{y}^{m-1}+q\right)_i=0; & \hbox{} \\
    \tilde{y}_i\geq 0, ~\left(\bar{\A}\tilde{y}^{m-1}+q\right)_i\geq 0,~\tilde{y}_i\left(\bar{\A}\tilde{y}^{m-1}+q\right)_i=0. & \hbox{}
  \end{array}
\right.
\end{equation} As $\hat{y}\neq \tilde{y}$, $\max\limits_{i}\{|\hat{y}_i^{m-1}-\tilde{y}_i^{m-1}|\}>0$.
Let $j:=\arg\max\limits_{i}\{|\hat{y}_i^{m-1}-\tilde{y}_i^{m-1}|\}$, and without loss of generality,
$\hat{y}_j-\tilde{y}_j>0$.
By direct calculation, we have
\begin{eqnarray}
  &&(\hat{y}_j-\tilde{y}_j)\left(\bar{\A}\hat{y}^{m-1}-\bar{\A}\tilde{y}^{m-1}\right)_j  \nonumber \\
  &=& (\hat{y}_j-\tilde{y}_j)\left(\bar{\A}\hat{y}^{m-1}+q-\bar{\A}\tilde{y}^{m-1}-q\right)_j \nonumber\\
  &=& - \hat{y}_j \left(\bar{\A}\tilde{y}^{m-1}+q\right)_j- \tilde{y}_j \left(\bar{\A}\hat{y}^{m-1}+q\right)_j\nonumber\\
  &\leq & 0.\nonumber
\end{eqnarray}
On the other hand, by the imposed conditions on the entries of $\A$,
\begin{eqnarray}
  &&(\hat{y}_j-\tilde{y}_j)\left(\bar{\A}\hat{y}^{m-1}-\bar{\A}\tilde{y}^{m-1}\right)_j  \nonumber \\
  &=& (\hat{y}_j-\tilde{y}_j)\left [\bar{a}_{j\cdots j}\left(\hat{y}_j^{m-1}-\tilde{y}_j^{m-1}\right)+ \sum\limits_{k\neq j}\bar{a}_{jk\cdots k}\left(\hat{y}_k^{m-1}-\tilde{y}_k^{m-1}\right)\right ]\nonumber\\
  &\geq &(\hat{y}_j-\tilde{y}_j)\left(\hat{y}_j^{m-1}-\tilde{y}_j^{m-1}\right) \left(\bar{a}_{j\cdots j}+\sum\limits_{k\neq j}\bar{a}_{jk\cdots k}\right)\nonumber\\
  &> & 0,\nonumber
\end{eqnarray} where the first inequality follows from the definition of $j$ and the fact that $\bar{\A}$ is a $Z$-tensor,
and the last inequality follows from (\ref{strict}). This is  a contradiction. Thus, TCP($\bar{\A},q$) has a unique solution, say $y^*$. We can easily  verify that $y^*$ is also the unique solution to the following problem:
$$ Dy\geq 0, \bar{\A}y^{m-1}+q \geq 0,\,\, \langle Dy, \bar{\A}y^{m-1}+q\rangle=0.$$
Invoking the definition of $\bar{\A}$, it follows readily that $Dy^*$ is the unique solution to TCP($\A,q$). This completes the proof.
\end{proof}

\gap

\noindent{\bf Concluding Remarks.} 
This paper is a revised version of an earlier submission. Here, certain results (e.g., Theorems 5.1 and 5.2) are modified, 
proofs  (e.g., of Theorem 3.1) are expanded, and references are updated. In closing, we cite
 two recent works related to tensor complementarity problems. In \cite{huang-qi},  
an $n$-person noncooperative game is formulated as a tensor complementarity problem. In \cite{gowda-pcp}, polynomial complementarity problems are studied, specifically extending Theorem 3.1 to polynomial functions. 
\end{document}